\documentclass[letterpaper, 10 pt, conference]{ieeeconf}  

\IEEEoverridecommandlockouts                              
\usepackage{mathtools}
\usepackage{amsmath,amssymb,amsfonts}
\usepackage{bm}
\usepackage[dvipsnames]{xcolor}

\usepackage{hyperref}
\usepackage{cancel}
\usepackage{array}
\usepackage{subcaption}
\usepackage{float}    
\usepackage{dblfloatfix}    
\DeclareCaptionLabelSeparator{periodspace}{.\quad}
\usepackage[nameinlink,capitalise]{cleveref}
\crefname{equation}{}{} 

\usepackage{standalone}
\usepackage{breakcites}
\usepackage{cite}
\let\labelindent\relax 
\usepackage{enumitem}
\newlist{identities}{enumerate}{1}
\setlist[identities,1]{
    label={\textbf{ID \arabic*}},
    ref=\arabic*, 
    wide,itemsep=0pt,topsep=0pt}
\creflabelformat{identitiesi}{#2\textup{#1}#3} 
\crefname{identitiesi}{ID}{IDs} 
\Crefname{identitiesi}{ID}{IDs}

\usepackage{multicol}

\usepackage{tikz}
\usetikzlibrary{spy,shapes,arrows,positioning,calc,decorations.markings}

\usepackage{pgfplots}
\pgfplotsset{compat=1.18}
\usepgfplotslibrary{colorbrewer} 
\usepgfplotslibrary{fillbetween}

\usepackage{pgfplotstable}
\usepackage{booktabs}
\usepackage{colortbl}

\newtheorem{proposition}{Proposition}
\newtheorem{definition}{Definition}

\newtheorem{theorem}{Theorem}
\newtheorem{remark}{Remark}

\crefname{lemma}{Lemma}{Lemmas} 
\Crefname{lemma}{Lemma}{Lemmas}

\newcommand*\circled[1]{\tikz[baseline=(char.base)]{\node(char)[shape=rounded rectangle,draw,inner sep=0.6pt,minimum height=1.5ex]{#1};}} %
\newcommand\kronF[2]{{#1}^{\circled{\tiny{\ensuremath{#2}}}}} 
\renewcommand\Vec[1][]{\textnormal{vec}\ensuremath{\if$#1$ \else \left[#1\right]\fi}}

\newcommand{\real}{\mathbb{R}}

\newcommand{\rd}{\text{\upshape d}} 

\newcommand{\bzero}{\ensuremath{\mathbf{0}}} 

\newcommand{\tv}{\ensuremath{\tilde{v}}}

\newcommand{\bA}{\ensuremath{\mathbf{A}}}
\newcommand{\bB}{\ensuremath{\mathbf{B}}}

\newcommand{\bD}{\ensuremath{\mathbf{D}}}

\newcommand{\bG}{\ensuremath{\mathbf{G}}}

\newcommand{\bI}{\ensuremath{\mathbf{I}}}

\newcommand{\bK}{\ensuremath{\mathbf{K}}}

\newcommand{\bQ}{\ensuremath{\mathbf{Q}}}
\newcommand{\bR}{\ensuremath{\mathbf{R}}}
\newcommand{\bS}{\ensuremath{\mathbf{S}}}

\newcommand{\bV}{\ensuremath{\mathbf{V}}}

\newcommand{\ba}{\ensuremath{\mathbf{a}}}

\newcommand{\be}{\ensuremath{\mathbf{e}}}
\renewcommand{\bf}{\ensuremath{\mathbf{f}}}
\newcommand{\bg}{\ensuremath{\mathbf{g}}}

\newcommand{\bq}{\ensuremath{\mathbf{q}}}

\newcommand{\bu}{\ensuremath{\mathbf{u}}}
\newcommand{\bv}{\ensuremath{\mathbf{v}}}

\newcommand{\bx}{\ensuremath{\mathbf{x}}}

\newcommand{\bz}{\ensuremath{\mathbf{z}}}

\newcommand{\cK}{\ensuremath{\mathcal{K}}}
\newcommand{\cL}{\ensuremath{\mathcal{L}}}

\newcommand{\perm}[2]{\ensuremath{\bS_{#1 \times #2}}}

\newcommand{\F}{\ensuremath{\mathbf{F}}}
\newcommand{\G}{\ensuremath{\mathbf{G}}}

\title{\LARGE \textbf{
    Computing Solutions to the Polynomial-Polynomial Regulator Problem*
}}

\author{Nicholas A. Corbin$^{1}$ and Boris Kramer$^{2}$
\thanks{*This work was supported by the National Science Foundation under Grant CMMI-2130727.}
\thanks{$^{1}$N. Corbin is with the Department of Mechanical and Aerospace Engineering, University of California San Diego, La Jolla, CA 92093-0411 USA {\tt\small ncorbin@ucsd.edu}}
\thanks{$^{2}$B. Kramer is with the Faculty of the Department of Mechanical and Aerospace Engineering, University of California San Diego, La Jolla, CA 92093-0411 USA {\tt\small bmkramer@ucsd.edu}}
}

\begin{document}
\maketitle

\begin{abstract}
    We consider the optimal regulation problem for nonlinear control-affine dynamical systems.
    Whereas the linear-quadratic regulator (LQR) considers optimal control of a \emph{linear} system with \emph{quadratic} cost function, we study \emph{polynomial} systems with \emph{polynomial} cost functions; we call this problem the polynomial-polynomial regulator (PPR).
    The resulting polynomial feedback laws provide two potential improvements over linear feedback laws: 1) they more accurately approximate the optimal control law, resulting in lower control costs, and 2) for \emph{some} problems they can provide a larger region of stabilization.
    We derive explicit formulas---and a scalable, general purpose software implementation---for computing the polynomial approximation to the value function that solves the optimal control problem.
    The method is illustrated first on a low-dimensional aircraft stall stabilization example,
    for which PPR control recovers the aircraft from more severe stall conditions than LQR control.
    Then we demonstrate the scalability of the approach on a semidiscretization of dimension $n=129$ of a partial differential equation,
    for which the PPR control reduces the control cost by approximately 75\% compared to LQR for the initial condition of interest.
\end{abstract}

\section{Introduction}\label{sec:intro}
Optimal control for linear dynamical systems has been well studied for decades.
This is evidenced by the availability of simple software tools that allow practitioners to design optimal controllers with just a few lines of code, such as \textsc{Matlab}'s $\texttt{lqr()}$.
These software tools have also been scaled to work for large-scale dynamical systems \cite{Benner2008a}, such as those arising from semidiscretization of partial differential equations (PDEs).
While nonlinear optimal control theory remains an active area of research, its adoption to large-scale nonlinear systems remains a major computational challenge.
The development of scalable software tools thus remains another active research direction.

A computational challenge in nonlinear optimal control theory involves solving the Hamilton-Jacobi-Bellman (HJB) PDE for the \emph{value function}.
Many approaches exist for approximating the solutions to HJB PDEs, including state-dependent Riccati equations~\cite{Cimen2012}, algebraic Gramians \cite{Condon2005,Gray2006,Benner2024}, discretization techniques \cite{Falcone2016}, iterative approaches \cite{Kalise2018,Dolgov2021}, and many others.
In this paper, we expand on a recent body of work \cite{Breiten2019,Borggaard2020,Borggaard2021,Almubarak2019,Kramer2024} bringing renewed interest to the method of Al'brekht \cite{Albrekht1961}, which is based on Taylor series expansions.
This approach to solving HJB PDEs has been popular ever since it was introduced in the 1960s;
however, since Al'brekht's results were presented in an abstract manner without closed-form solutions, the method historically was only applied to models with a few state dimensions and simplified dynamics in order to make the computations manageable to carry out \cite{Garrard1967,Garrard1972}.
Scalable software tools automating these computations
are required for general purpose use of Al'brekht's method, e.g. for larger state dimensions and for models with more than one or two nonlinearities.

Krener's Nonlinear Systems Toolbox (NST), originally introduced in 1997, was, to the authors' knowledge, the first openly available implementation of Al'brekht's method for general purpose use \cite{Krener2019}.
However, certain symbolic computations used in NST hindered its scalability, as detailed in \cite{Borggaard2018}.
Borggaard and Zietsman went on to provide an implementation based on the Kronecker product in \cite{Borggaard2020,Borggaard2021}; a similar approach was simultaneously presented by Almubarak et al. \cite{Almubarak2019}.
These approaches were shown to scale well to state-dimensions on the order of $n=20$ to $n=40$; however, these works only consider systems with polynomial drift, linear inputs, and quadratic cost functions.
There are, however, applications requiring HJB PDE solutions for systems with polynomial input maps and even polynomial state-dependence in the cost function; see, for example, nonlinear balanced truncation model reduction \cite{Scherpen1993,Krener2008,Fujimoto2008a,Kramer2023}.

The major contributions of this paper are to derive closed-form formulas, and existence of solution proofs, for value function approximations for systems with:
\begin{enumerate}
    \item polynomial state-dependence in the drift \emph{and} input map;
    \item polynomial state-dependence in the state penalty in the cost function.
\end{enumerate}
In addition to these two theoretical contributions, we provide an accompanying open-source, scalable implementation of the proposed algorithms for practical use.
The function \texttt{ppr()} in the \texttt{cnick1/PPR} repository \cite{Corbin2024} acts as a nonlinear analog to \textsc{Matlab}'s \texttt{lqr()}.

\section{Preliminaries}\label{sec:background}
We review relevant
optimal control theory in \cref{sec:optimal-control},
followed by a summary of Al'brekht's method in \cref{sec:albrekht-method}.
Kronecker product definitions and identities are then reviewed in \cref{sec:kronecker}.
\subsection{Optimal Control Theory}\label{sec:optimal-control}
Consider the control-affine dynamical system
\begin{equation}\label{eq:FOM-NL}
    \dot{\bx}(t) = \bf(\bx(t))  + \bg(\bx(t)) \bu(t),
\end{equation}
where: $t$ is time,
$\bx(t) \in \real^n$ is the state,
$\bu(t) \in \real^m$ is the input,
$\bf \colon \real^n \to \real^n$ is the drift, and
$\bg \colon \real^n \to \real^{n \times m}$ is the input map.
The optimal control problem for \cref{eq:FOM-NL} seeks to find an input signal $\bu(t)$
that minimizes the infinite-horizon scalar cost function
\begin{equation} \label{eq:performanceIndex}
    J(\bx_0,\bu)  \coloneqq \frac{1}{2} \int_{0}^{\infty} \left(\bx^\top \bm{Q}(\bx) \bx +  \bu^\top \bm{R}(\bx) \bu \right)\rd t,
\end{equation}
where $\bm{Q}(\bx)\succeq \bzero$ and $\bm{R}(\bx)\succ \bzero$ are nonnegative definite and positive definite symmetric matrix-valued functions, respectively, of appropriate dimensions.
\begin{definition}
    The value the cost function takes under the action of the optimal control is given by \emph{the value function}:
    \begin{align} \label{eq:valueFunction}
        V(\bx_0) & \coloneqq \min_{\bu}  J(\bx_0,\bu).
    \end{align}
\end{definition}
The optimal control is denoted $\bu_*$, so $V(\bx_0)  \equiv J(\bx_0,\bu_*) $.
The next theorem summarizes the well-known result that the solution to the optimal control problem can be obtained by solving the HJB PDE.
\begin{theorem}[e.g. \cite{Kalman1963,Lewis2012}]
    Assume that the cost \cref{eq:performanceIndex} is continuously differentiable in all of its arguments and is strictly convex in $\bu$.
    Then the value function is the solution to the HJB PDE
    \begin{equation}
        \begin{split}\label{eq:HJB-PDE}
            0 & = \frac{\partial V^\top(\bx)}{\partial \bx}  \bf(\bx) - \frac{1}{2} \frac{\partial V^\top(\bx)}{\partial \bx}\bg(\bx) \bm{R}^{-1}(\bx)\bg^\top(\bx) \frac{\partial V(\bx)}{\partial \bx} \\
              & \qquad + \frac{1}{2} \bx^\top \bm{Q}(\bx) \bx.
        \end{split}
    \end{equation}
    Furthermore, the optimal control $\bu_*$ is given in feedback form by the gradient of the value function
    as
    \begin{equation}\label{eq:optimal-u}
        \bu_*(\bx) =  - \bm{R}^{-1}(\bx) \bg^\top(\bx) \frac{\partial V(\bx)}{\partial \bx}.
    \end{equation}
\end{theorem}
Assuming a solution exists that satisfies the HJB PDE \cref{eq:HJB-PDE}, then the optimal control is given by \cref{eq:optimal-u}; hence, the optimal control problem reduces to solving the HJB PDE \cref{eq:HJB-PDE}.

\subsection{Al'brekht's Method}\label{sec:albrekht-method}
Computing solutions to the HJB PDE \cref{eq:HJB-PDE} is nontrivial and, in general, not possible analytically.
In the interest of developing scalable algorithms, we adopt the approach of Al'brekht \cite{Albrekht1961} to compute polynomial approximations to the value function.
Al'brekht's method has three main features, which we summarize in the next theorem.
\begin{theorem}[Al'brekht's method \cite{Albrekht1961,Lukes1969}]\label{thm:albrekht}
    Assume that the functions $\bf(\bx)$ and $\bg(\bx)$ in the dynamics, along with the functions $\bm{Q}(\bx)$ and $\bm{R}(\bx)$ in the cost function, are analytic.
    Also assume that a stabilizing solution exists to the LQR problem associated with the
    linearized dynamics\footnote{In practice this is the main limitation/assumption for the method: the algebraic Riccati equation for the LQR problem must have a solution.}.
    Then:
    \begin{enumerate}
        \item the value function
              is analytic and can be approximated by a degree~$d$ polynomial;
        \item the lowest degree polynomial term in the value function is degree~$2$, whose polynomial coefficient
              is given by the solution to the algebraic Riccati equation associated with the LQR problem on the linearized dynamics;
        \item the remaining higher degree polynomial coefficients of the value function are solutions to
              linear algebraic equations that are entirely determined by the already-computed polynomial coefficients.
    \end{enumerate}
\end{theorem}

The remainder of this paper will focus on computing approximate solutions to the HJB PDE \cref{eq:HJB-PDE} using Al'brekht's method; the next section introduces Kronecker product notation to aid in that task.

\subsection{Kronecker Product Definitions and Notation}\label{sec:kronecker}
The Kronecker product of two matrices $\bA \in \real^{p \times q}$ and $\bB \in \real^{s \times t}$ is the $ps \times qt$ block matrix
\begin{align*}
    \bA \otimes \bB \coloneqq \begin{bmatrix} a_{11}\bB & \cdots & a_{1q}\bB \\
                \vdots    & \ddots & \vdots    \\
                a_{p1}\bB & \cdots & a_{pq}\bB
                              \end{bmatrix},
\end{align*}
where $a_{ij}$ denotes the $(i,j)$th entry of $\bA$.
We write repeated Kronecker products as
\begin{equation*}
    \kronF{\bx}{k} \coloneqq \underbrace{\bx \otimes \dots \otimes \bx}_{k \ \text{times}}
    \in \real^{n^k}.
\end{equation*}

\noindent For $\bA \in \real^{p \times q}$, we define the
\textit{$k$-way Lyapunov matrix} as
\small
\begin{equation*}
    \cL_k(\bA) \coloneqq \sum_{i=1}^k\underbrace{\bI_p \otimes \bA \otimes \bI_p \otimes \dots \otimes \bI_p}_{\text{$k$ factors, $\bA$ in the $i$th position}} \in \real^{p^k \times p^{k-1}q}.
\end{equation*}
\normalsize
We also use the $\Vec{[\cdot]}$ operator, which stacks the columns of a matrix into one tall column vector, and the \textit{perfect shuffle matrix} $\perm{q}{p}$ \cite{Henderson1981,VanLoan2000}, defined as the permutation matrix which shuffles $\Vec[\bA]$ to match $\Vec[\bA^\top]$:
\begin{equation}\label{eq:vecPermutation}
    \Vec[\bA^\top]=\perm{q}{p} \Vec[\bA].
\end{equation}

A concept which arises when dealing with Kronecker product polynomials is symmetry of the coefficients (a generalization of symmetry of a matrix), as defined next.
\begin{definition}[Symmetric coefficients\label{def:sym}] Given a homogeneous polynomial of the form $\bv_d^\top \kronF{\bx}{d}$, the coefficient $\bv_k \in \real^{n^k \times 1}$ is \emph{symmetric} if for all $\ba_i \in \real^n$ it satisfies
    \begin{align*}
        \bv_k^\top \left(\ba_1 \otimes \ba_2 \otimes \cdots \otimes \ba_k\right) = \bv_k^\top \left(\ba_{i_1} \otimes \ba_{i_2} \otimes \cdots \otimes \ba_{i_k}\right),
    \end{align*}
    where the indices $\{ i_j \}_{j=1}^k$ are any permutation of $\{1, \dots, k \}$.
\end{definition}

From \cref{eq:vecPermutation}, one can see that if the matrix $\bA$ is symmetric, the vector $\Vec[\bA]$ is invariant under certain permutations.
The next \lcnamecref{thm:symPermutation} formalizes this concept in terms of the perfect shuffle matrix and the definition of symmetry in \cref{def:sym}.
\begin{proposition}[Permutation of symmetric coefficients]\label{thm:symPermutation}
    If a coefficient $\bv_k \in \real^{n^k \times 1}$ is symmetric as per \cref{def:sym}, then
    \begin{align*}
        \bv_k & = \perm{n^j}{n^i} \bv_k \qquad \forall i,j\geq0 \quad \text{s.t.}\quad  i+j=k.
    \end{align*}
\end{proposition}

\cref{tab:identities} in \cref{sec:identities} provides a collection of additional Kronecker product identities compiled from \cite{Brewer1978,Henderson1981,VanLoan2000,Magnus2019}.

\section{Kronecker Polynomial-Based Value Function Computations} \label{sec:NLBT-Poly}
For computational purposes, we
consider a nonlinear control-affine dynamical system with polynomial structure
\begin{equation}\label{eq:FOM-Poly}
    \begin{split}
        \dot{\bx} & = \underbrace{\bA \bx + \sum_{p=2}^\ell \F_p \kronF{\bx}{p}}_{\bf(\bx)} + \underbrace{\left(\sum_{p=1}^\ell \G_p \left(\kronF{\bx}{p} \otimes \bI_m\right) + \bB\right)}_{\bg(\bx)} \bu, \\
    \end{split}
\end{equation}
where $\bA \in \real^{n\times n}$, $\F_p \in \real^{n\times n^p}$, $\bB \in \real^{n\times m}$, and $\G_p \in \real^{n \times mn^p}$.
Often, the model of interest \cref{eq:FOM-NL} is already polynomial, in which case system \cref{eq:FOM-Poly} is an exact representation.
For example, many common PDEs, including Navier-Stokes, Kuramoto-Sivashinsky, Burgers, Allen-Cahn, Korteweg-de Vries, and Fokker-Planck all feature polynomial nonlinearities;
upon spatial discretization, these all yield systems of the form \cref{eq:FOM-Poly}.
If system \cref{eq:FOM-NL} is not exactly polynomial but is analytic, then system \cref{eq:FOM-Poly} is its Taylor approximation.
We also assume the cost \cref{eq:performanceIndex} to be analytic so that the functions $\bm{Q}(\bx)$ and $\bm{R}(\bx)$ can be expanded as real convergent power series.
In this article, we will only consider state dependence in $\bm{Q}(\bx)$ and drop the state-dependence of $\bm{R}(\bx)$; in theory it is possible, but not trivial, to include it.
Then we write the cost in terms of the Kronecker product as
\begin{align}
    J(\bx_0,\bu)  \coloneqq \frac{1}{2} \int_{0}^{\infty} \left(\bx^\top \bQ \bx  +  \bu^\top \bR \bu + \sum_{p=3}^{\lambda} \bq_p^\top \kronF{\bx}{p} \right)\rd t, \label{eq:performanceIndex-Poly}
\end{align}
where $\bq_p \in \real^{n^p}$
and we assume $\bQ \succeq \bzero$ and $\bR \succ \bzero$.
In many cases, the cost function chosen by the control engineer is already polynomial and this representation is exact; otherwise, it can be viewed as a Taylor approximation.

Following the result of \cref{thm:albrekht}, the value function can be approximated as a degree~$d$ polynomial; in this work, we write this explicitly using the Kronecker product as
\begin{align}
    V(\bx) & \approx \frac{1}{2} \bx^\top \bV_2 \bx + \frac{1}{2} \sum_{i=3}^d \bv_i^\top \kronF{\bx}{i}, \label{eq:vi-coeffs}
\end{align}
with the coefficients $\bv_i \in \real^{n^i}$ for $i=2, 3, \dots, d$.
The next theorem provides our main result, which is the explicit equations for the coefficients $\bv_i$ in Kronecker product form.

\begin{theorem}[Value function coefficients]\label{thm:wiPoly}
    Let the value function $V(\bx)$, which solves the HJB PDE \cref{eq:HJB-PDE} for the polynomial system \cref{eq:FOM-Poly}, be of the form~\cref{eq:vi-coeffs} with the coefficients $\bv_i \in \real^{n^i}$ for $i=2,3,\dots,d$.
    Then $\bv_2 = \Vec\left[\bV_2\right]$, where $\bV_2$ is the symmetric positive semidefinite solution to the algebraic Riccati equation
    \begin{align}\label{eq:ARE}
        \bzero  = & \bA^\top \bV_2 + \bV_2 \bA - \bV_2 \bB \bR^{-1} \bB^\top \bV_2 + \bQ.
    \end{align}
    For $3 \leq k \leq d$, let $\mathbf{\tv}_k \in \real^{n^k}$ solve the linear system
    \small
    \begin{equation}\label{eq:LinSysForWk-Poly}
        \begin{split}
             &
            \cL_{k} \left(\bA - \bB \bR^{-1} \bB^\top  \bV_2 \right)^\top
            \mathbf{\tv}_k =
            - \sum_{\mathclap{\substack{i,p\geq 2                                                                                                     \\ i + p = k+1}}} \cL_i(\F_p)^\top \bv_i  \\
             & \hspace{.4cm} -   \bq_k + \frac{1}{4}\!\!\!\!\sum_{\substack{i,j>2                                                                     \\ i+j=k+2}} \!\!\!\!ij~{\Vec}(\bV_i^\top \bB \bR^{-1} \bB^\top \bV_j)
            \\
             & \hspace{.4cm} + \frac{1}{4} \sum_{o=1}^{2\ell}\left( \sum_{\substack{p,q \geq 0                                                        \\p+q=o}}   \left(  \sum_{\substack{i,j\geq 2 \\       i+j=k-o+2}} \!\!\!\!ij~\Vec \Biggl[ \left(\bI_{n^p} \otimes \Vec\left[ \bI_m\right]^\top\right) \right. \right.\\
             & \hspace{.4cm}\times \left(\Vec\left[\G_q^\top \bV_j \right]^\top \otimes \left(\G_p^\top \bV_i \otimes \bR^{-1} \right) \right) \times \\
             & \hspace{0.25cm} \left.\left.\left(\bI_{n^{j-1}} \otimes \perm{n^{i-1}}{n^q m} \otimes \bI_m \right)
                    \left( \bI_{n^{k-p}} \otimes \Vec\left[ \bI_m\right] \right)\Biggr]
            \vphantom{\sum_{\substack{p,q \geq 0                                                                                                      \\p+q=o}}} \right)\right)
        \end{split}
    \end{equation}
    \normalsize
    Then the coefficient vector $\bv_k = \Vec\left(\bV_k\right) \in \real^{n^k}$ for $3 \leq k \leq d$ is obtained by symmetrization of $\mathbf{\tv}_k$.
\end{theorem}

The proof of \cref{thm:wiPoly} can found in \cref{sec:thm-proof}; it consists of inserting the polynomial forms for the $\bf(\bx)$, $\bg(\bx)$, $\bm{Q}(\bx)$, and $V(\bx)$ into the HJB PDE \cref{eq:HJB-PDE}.
Collecting terms of the same polynomial degree leads to algebraic equations for each of the unknown value function coefficients $\bv_i \in \real^{n^i}$ for $i=2,3,\dots,d$.

\begin{theorem}[Existence and uniqueness of solutions]
    Under the assumptions of \cref{thm:albrekht}, the linear system \cref{eq:LinSysForWk-Poly} has a unique solution for each coefficient $\bv_k$ for $3 \leq k \leq d$.
\end{theorem}
\begin{proof}
    The assumptions of \cref{thm:albrekht} imply that $(\bA - \bB\bR^{-1}\bB^\top \bV_2 )$ is Hurwitz, which implies that $\cL_{k}(\bA - \bB\bR^{-1}\bB^\top \bV_2 )$ is nonsingular.
    Hence the linear systems have unique solutions.
\end{proof}

\begin{remark}\label{prop:flops}
    The cost of computing degree~$d$ approximations to the value function with \cref{thm:wiPoly} is $O(dn^{d+1})$ using our implementation, whereas a naive approach costs $O(n^{3d})$.
    The flop count is tedious, so we omit it here and refer the interested reader to our other publication \cite{Corbin2024c} where similar details can be found.
    To summarize, the main algorithmic accelerations come from reshaping Kronecker products to leverage Basic Linear Algebra Subprograms (BLAS) operations and the use of a specialized linear solver \cite{Borggaard2021,Chen2019} that takes advantage of the \emph{$k$-way Lyapunov structure} of the linear systems \cref{eq:LinSysForWk-Poly}.
\end{remark}

\begin{figure*}[b]
    \centering
    \hrulefill
    \vspace*{4pt}
    \includegraphics[width = \textwidth,page=1]{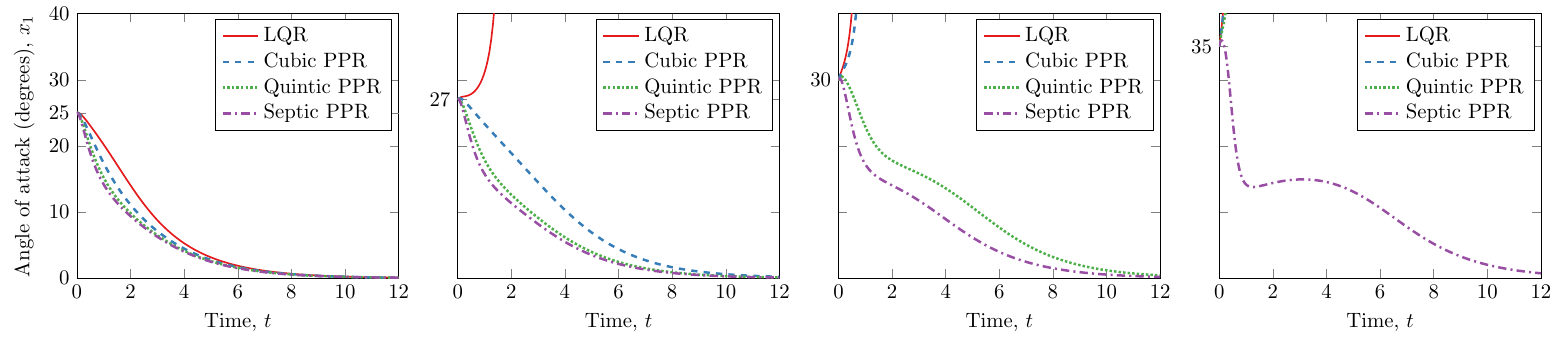}
    \caption{Angle of attack response for initial conditions (from left to right) $\alpha_0 = 25^{\circ}, 27^{\circ}, 30^{\circ}, 35^{\circ}$. The PPR controllers stabilize the aircraft faster than LQR, and they are able to recover the aircraft from stall for larger initial angles of attack.}
    \label{fig:example7}
\end{figure*}

\begin{remark}
    Given a degree $d$ \emph{approximation} to the value function computed using \cref{thm:wiPoly}, a degree $d-1$ approximation to the optimal control is given by \cref{eq:optimal-u}.
    This suboptimal polynomial feedback law has the form
    \begin{equation}\label{eq:poly-feedback}
        \bu(\bx) = \bK \bx + \cK^{[2]} \kronF{\bx}{2} + \dots + \cK^{[d-1]} \kronF{\bx}{d-1}.
    \end{equation}
    The linear coefficient $\bK$ is precisely the LQR gain matrix $\bK = - \bR^{-1}\bB \bV_2$.
    The higher-order gain matrices are obtained by collecting the terms of each degree from \cref{eq:optimal-u}, where the polynomial form of $\bg(\bx)$ is given in \cref{eq:FOM-Poly} and the gradient of the value function is given by \cref{eq:Past-value-deriv-compact}.
\end{remark}

\section{Numerical Examples}\label{sec:results}
We demonstrate the approach for computing value functions and the associated controllers outlined in \cref{thm:wiPoly} on two examples.
First, we consider an aircraft flight model with state dimension $n=3$ in \cref{sec:example1}.
Then, in \cref{sec:example2}, we demonstrate that the approach can be scaled to a semidiscretized model of the Allen-Cahn PDE with state dimension $n=129$.

\subsection{3D Example: Aircraft Stall Stabilization}\label{sec:example1}

We consider an aircraft model used in \cite{Garrard1977} to formulate a nonlinear stall-stabilization controller for an F-8 Crusader cruising at $30,000$ ft at $\text{Mach}=0.85$.
In that work, the authors use an Al'brekht-based approach to compute nonlinear regulators by hand because a scalable, automated approach had not been developed.
This tedious approach limited the authors to only consider cubic drift nonlinearities and no nonlinear input terms.
Furthermore, they only consider quadratic-in-state costs.
This problem was also considered in \cite{Almubarak2019}, but they also discard the nonlinear input terms.
In contrast, our work provides the general purpose software to automate the computation so practitioners can easily implement nonlinear controllers including these terms.

\subsubsection*{Model}
The derivation of the state-space model from Newton's second law can be found in \cite{Garrard1977}.
The result is a state-space model in terms of the angle of attack $x_1$, the angle of the plane relative to the trim pitch $x_2$, and the rotation rate of the plane $x_3$, with the control input $u$ corresponding to the angle of the tail elevator:
\begin{equation}\label{eq:garrard-model}
    \begin{split}
        \dot{x}_1 & = x_3 - x_1^2 x_3 - 0.088 x_1 x_3 - 0.877 x_1 + 0.47 x_1^2 \\
                  & \quad - 0.019 x_2^2 + 3.846 x_1^3 - 0.215 u + 0.28 u x_1^2 \\
        \dot{x}_2 & = x_3                                                      \\
        \dot{x}_3 & = -0.396 x_3 - 4.208 x_1 - 0.47 x_1^2 - 3.564 x_1^3        \\
                  & \quad - 20.967 u + 6.265 ux_1^2.
    \end{split}
\end{equation}
\subsubsection*{Control Problem}
We apply our algorithm to the same control problem presented in \cite{Garrard1977}:
the model is subjected to disturbances in the angle of attack, physically corresponding to a gust of wind that puts the aircraft into a stall condition, which is reported in \cite{Garrard1977} to occur at an angle of $\alpha = 23.5^{\circ}$.
Hence given an initial condition $\bx_0 = \begin{bmatrix} \alpha_0 & 0 & 0 \end{bmatrix}^\top$, the objective is to design a controller that causes the plane to recover from stall.
This is formulated as an optimal control problem where we seek a control $\bu$ that minimizes the cost function
\begin{align*}
    J(\bx_0, \bu) = \frac{1}{2} \int_0^\infty \left(\bx^\top \bQ \bx + \bu^\top \bR \bu\right) \,\rd t
\end{align*}
with $\bQ = \frac{1}{4} \bI$ and $\bR = \bI$ subject to the dynamics \cref{eq:garrard-model}.

\subsubsection*{Results}
We compute degree $2$, $4$, $6$, and $8$ approximations to the value function, giving degree $1$, $3$, $5$ and $7$ controllers of the form \cref{eq:poly-feedback}, denoted LQR, Cubic PPR, Quintic PPR, and Septic PPR, respectively.
\cref{fig:example7} shows the angle of attack response under the action of the different controllers for initial conditions of $\alpha_0 = 25^{\circ}, 27^{\circ}, 30^{\circ}, 35^{\circ}$.
For the smallest angle of attack, $\alpha_0 = 25^{\circ}$, all controllers---even the LQR---successfully stabilize the aircraft.
However, as shown in \cref{tab:valueFunApproxs-2}, the PPR controllers exhibit lower control costs, demonstrating the improved efficiency of polynomial feedback laws over linear feedback laws.
\begin{table}[htb]
    \centering
    \caption{Aircraft control costs ($\alpha_0 = 25^{\circ}$) computed up to $T = 12$. PPR controllers have a lower cost.}
    \begin{tabular}{ccc}
        \toprule
        Controller  & $ \frac{1}{2} \int_0^T \left(\bx^\top \bQ \bx + \bu^\top \bR \bu \right)\,\rd t$ \\
        \midrule
        LQR         & 0.053166                                                                         \\
        Cubic PPR   & 0.044503                                                                         \\
        Quintic PPR & 0.040593                                                                         \\
        Septic PPR  & 0.039393                                                                         \\
        \bottomrule
    \end{tabular}
    \label{tab:valueFunApproxs-2}
\end{table}

Regarding the ability to recover the aircraft from stall, we see a graceful degradation of the controllers.
As the angle of attack increases, we gradually see each controller fail while the higher-order controllers are still able to stabilize the aircraft.
The higher-order controllers achieve their improved performance by providing more rapid control inputs, so practitioners need to consider this when designing the controllers, i.e. picking the weights $\bQ(\bx)$ and $\bR$.

This simple example demonstrates two reasons to consider polynomial feedback laws: 1) the control costs can be lower, and 2) PPR controllers may work in cases where LQR fails. There are cases where LQR works while PPR fails though, so practitioners should always exercise caution.
\subsection{Allen-Cahn Equation}\label{sec:example2}
We consider a semidiscretization of the Allen-Cahn PDE
\begin{align*}
    \frac{\partial w(z,t)}{\partial t} & = \epsilon \frac{\partial^2 w(z,t)}{\partial z^2} + w(z,t) - w(z,t)^3
\end{align*}
for $t > 0$ and $z \in \Omega = [-1,1]$.
The system is subject to the boundary conditions $w(-1,t)=-1$ and $w(1,t)=1$,
and the initial condition $w(z,0) = 0.53 z + 0.47\sin(-1.5\pi z)$ is chosen based on Example 34 in \cite{Trefethen2000}.
This reaction-diffusion PDE models phase separation in a two-phase mixture.
The solution describes the evolution of interfaces between the two phases, with
the diffusion coefficient $\epsilon$ determining the interface thickness.
The system has the three constant equilibrium solutions $w_{\text{ss}}(z) = -1,0,1$.
The solutions $w_{\text{ss}}(z) = \pm 1$ are stable and represent a pure single phase,
whereas $w_{\text{ss}}(z) = 0$ is unstable and represents a homogeneous mixture of the two phases.
The PDE also has a family of nontrivial steady-state solutions given by
\begin{align}
    w_{\text{ss}}(z; \epsilon, z_0) = \tanh\left(\frac{z-z_0}{\sqrt{2 \epsilon}} \right), \label{eq:steady-state-soln}
\end{align}
where $z_0$ is the position of the interface.
This is the only solution that satisfies the boundary conditions, so our problem is expected to converge to this type of solution.

\subsubsection*{Finite-Dimensional Model}
To put the model in the form \cref{eq:FOM-Poly}, the PDE is spatially discretized using a Chebychev pseudospectral collocation method with $n=129$ nodes
\cite[Ex.~34]{Trefethen2000},
as depicted in \cref{fig:cheb-nodes}.
The model is also augmented with three control inputs, as shown in \cref{fig:cheb-nodes}.
The Allen-Cahn PDE has polynomial structure, so the only approximation comes from the finite element discretization.
\begin{figure}[htb]
    \centering
    \includegraphics[page=3,width=\columnwidth]{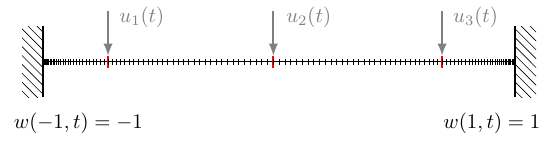}
    \caption{Spatial domain and control locations for the Allen-Cahn PDE discretized with 129 Chebychev nodes, which are denser near the boundaries.}
    \label{fig:cheb-nodes}
\end{figure}

The semidiscretized model takes the form
\begin{align}
    \dot\bx & = \bA \bx + \F_3 \kronF{\bx}{3} + \bB \bu. \label{eq:allen-cahn-1}
\end{align}
The state $\bx(t) = \begin{bmatrix} w(z_1,t),\dots,w(z_{n},t)\end{bmatrix}^\top \in \real^{n}$ represents the solution $w(z,t)$ evaluated at the Chebychev nodes $\bz = \begin{bmatrix} z_1, \dots, z_{n} \end{bmatrix}^\top$.
The linear component of the drift is given by $\bA = \epsilon \texttt{D}_\texttt{z}^2 + \bI$, where $\texttt{D}_\texttt{z}^2$ is the differentiation matrix approximating the operator $\partial^2 /\partial z^2$.
The cubic drift coefficient $\F_3$ is a sparse binary matrix with ones in the positions to satisfy
$\F_3 \kronF{\bx}{3} = \bx \odot \bx \odot \bx $, where $\odot$ is the Hadamard product (element-wise multiplication).
Three independent control inputs are placed at nodes $33$, $65$, and $97$, so the input matrix is $\bB = \begin{bmatrix}
        \be_{33} & \be_{65} & \be_{97}
    \end{bmatrix}$, where $\be_i \in \real^n$ is the $i$th standard basis vector.

\cref{fig:example9-openLoop} shows the open-loop behavior of the system for $\epsilon = 0.01$.
As described in \cite{Trefethen2000}, the system first exhibits a ``metastable'' configuration with three interfaces before suddenly transitioning to a steady-state solution of the form \cref{eq:steady-state-soln} with the interface at around $z_0 = 0$ at $t \approx 40$.
\begin{figure}[htb]
    \centering
    \includegraphics[width=0.9\columnwidth]{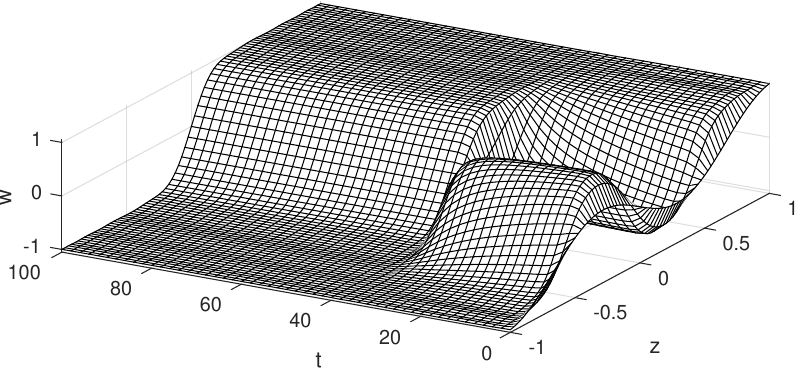}
    \caption{Open-loop behavior of the Allen-Cahn example for $\epsilon = 0.01$, as shown in \cite[Ex.~34]{Trefethen2000}.}
    \label{fig:example9-openLoop}
\end{figure}

\subsubsection*{Control Problem}
Our objective is to dictate the location $z_0$ of the interface between the two phases in the steady-state solution.
This is formulated as an optimal control problem seeking to asymptotically stabilize an equilibrium of the form \cref{eq:steady-state-soln} for a desired interface location $z_0$.
The equilibrium at the origin $\bx = \bzero$ of \cref{eq:allen-cahn-1} corresponds to the unstable trivial solution $w_{\text{ss}}(z) = 0$, which does not satisfy the boundary conditions,
so we shift the desired reference solution to the origin using a change of coordinates
$\bar\bx = \bx - \bx_{\text{ref}}$, where $\bx_{\text{ref}} = \begin{bmatrix} w_{\text{ss}}(z_1; \epsilon, z_0),\dots,w_{\text{ss}}(z_{n}; \epsilon, z_0)\end{bmatrix}^\top $ corresponds to a reference equilibrium solution defined in \cref{eq:steady-state-soln} for a particular interface location $z_0$.
The cost function we choose is
\begin{align*}
    J(\bar\bx_0, \bu) = \frac{1}{2} \int_0^\infty \left(\bar\bx^\top \bQ \bar\bx  + \bu^\top \bR \bu + \bq_4^\top \kronF{\bar\bx}{4}\right)\,\rd t
\end{align*}
with
$\bQ = 0.1 \bI$,
$\bR = \bI$,
and
$\bq_4$ the sparse vector satisfying $\bq_4^\top \kronF{\bar\bx}{4} =  \sum_{i=1}^n \bar{x}_i^4 $ (analgous to the identity matrix for $\bQ$).

\subsubsection*{Results}
Selecting the desired interface location as $z_0 = 0.5$, we compute a standard LQR controller, a Quadratic PPR controller, and a Cubic PPR controller to stabilize the system to the desired solution.
In addition to the diffusion parameter value $\epsilon = 0.01$ from \cite{Trefethen2000}, we also include results for the cases $\epsilon = 0.0075$ and $\epsilon = 0.005$.
The smaller the diffusion coefficient, the more the cubic nonlinearity dominates the dynamics; this has the effect of making the metastable configuration more persistent and making it more difficult to stabilize the desired equilibrium.

Closed-loop simulations of the nonlinear system \cref{eq:allen-cahn-1} are performed using
\textsc{Matlab}'s \texttt{ode23s()} up to $T=1000$.
We then evaluate the cost function to quantitatively assess each of the controllers;
the results are shown in \cref{tab:valueFunApproxs-example9}.

\begin{table}[htb]
    \centering
    \caption{Cost $\frac{1}{2} \int_0^T \left(\bar\bx^\top \bQ \bar\bx  + \bu^\top \bR \bu + \bq_4^\top \bar{\bx}^{\tiny\textcircled{4}}\right)\,\rd t$ integrated to $T = 1000$ for the Allen-Cahn example.}
    \begin{tabular}{cccc}
        \toprule
        Controller    & $\epsilon=0.01$ & $\epsilon=0.0075$ & $\epsilon=0.005$ \\
        \midrule
        LQR           & 5475.640        & 19376.855         & 87268.670        \\
        Quadratic PPR & 4339.483        & 14042.908         & 57876.913        \\
        Cubic PPR     & 1372.454        & 4153.668          & 20711.449        \\
        \bottomrule
    \end{tabular}
    \label{tab:valueFunApproxs-example9}
\end{table}

For the parameter values considered, the Cubic PPR controller has a cost about
75\% lower
than the LQR controller.
As a rule of thumb, controllers of even degree are discouraged, since the associated value function (which locally acts as a Lyapunov function for the closed-loop dynamics) would be of odd degree, which is ill-advised.
Nonetheless, we have included the results for a Quadratic PPR controller for comparison purposes; it also performs better than LQR for this initial condition and these parameter values, but the improvement is not as significant as the Cubic PPR controller.
Note that neither the Quadratic PPR nor the Cubic PPR controller solves the problem exactly; an infinite Taylor series would be required, and the solution is only guaranteed to converge locally.
Still, including higher-order polynomial terms from the dynamics and the cost function noticably improves the performance of the controllers.
In particular, the extra term in the cost function penalizes deviations from the reference configuration more heavily, so the PPR controllers are expected to stabilize the system more quickly and avoid the metastable configuration present in the open-loop solution.

\cref{fig:example9-linearControl} and \cref{fig:example9-cubicControl} show the closed-loop system under the action of the LQR controller and the Cubic PPR controller, respectively,
for the $\epsilon= 0.01$ case whose open-loop behavior is shown in \cref{fig:example9-openLoop}.
While both controllers stabilize the system to the desired solution, the Cubic PPR controller is able to perform the same task much more rapidly.

\begin{figure}[htb]
    \centering
    \begin{subfigure}[]{.9\columnwidth}
        \centering
        \includegraphics[width=\columnwidth]{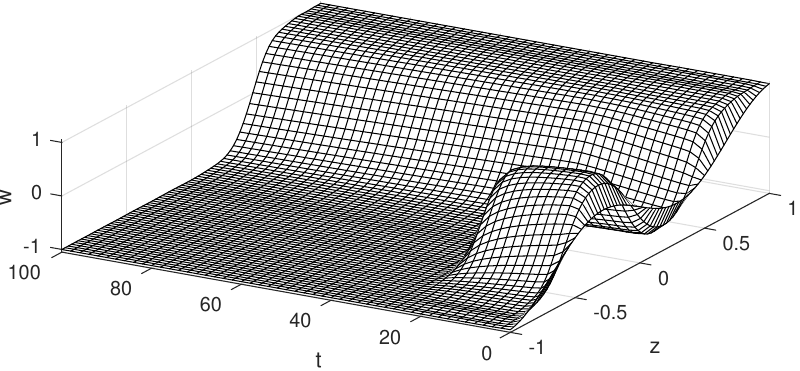}
        \caption{LQR}
        \label{fig:example9-linearControl}
    \end{subfigure}

    \begin{subfigure}[]{.9\columnwidth}
        \centering
        \includegraphics[width=\columnwidth]{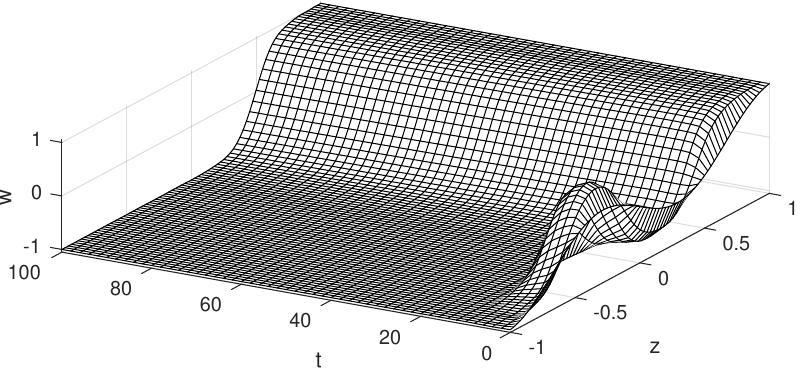}
        \caption{Cubic PPR}
        \label{fig:example9-cubicControl}
    \end{subfigure}
    \caption{Allen-Cahn example closed-loop simulations for $\epsilon = 0.01$}.
    \label{fig:example9}
\end{figure}

\section{Conclusions And Future Work}
In this paper, we provided a general scalable approach and open-access software for computing solutions to the polynomial-polynomial regulator problem.
The main contributions of this work include the development of algorithms capable of handling general polynomial state-dependence in drift, input map, and in the state penalty in the cost function.
These contributions enable more accurate local approximations of optimal control laws,
and the capability to include polynomial terms in the cost function gives practitioners more flexibility to tune controller behavior.

The results for the two examples considered herein demonstrate some of the potential benefits of polynomial controllers over linear controllers:
\begin{enumerate}
    \item PPR controllers may have lower control costs than LQR for initial conditions sufficiently close to the origin;
    \item PPR controllers may be able to stabilize initial conditions that LQR fails to stabilize.
\end{enumerate}
However, these controllers also experience the following limitations, since they are based on Taylor expansions:
\begin{enumerate}
    \item the solutions are only guaranteed to be valid within a neighborhood of the origin;
    \item the region of attraction in the closed-loop system does not necessarily increase as more terms are added to the controller.
\end{enumerate}
Practically speaking, there exist systems that are globally stabilized by LQR but only locally stabilized by PPR, so
practitioners seeking the advantages of polynomial controllers should be wary of their limitations as well.

As future work, we plan to investigate further the practical role of higher-order terms in the cost function.
While the $\bQ$ and $\bR$ matrices in LQR are well understood for trading controller performance for efficiency, the significance of higher-order terms such as $\bq_4$ in the cost function is not yet well understood, in particular regarding the global behavior of the controllers.
Additionally, the tensor structure of the higher-order value function and feedback gain coefficients is not fully understood.
They are known to often have low numerical rank which can be exploited by some numerical methods, but a major limitation that remains is the RAM requirements for storing the higher-order coefficients. We plan to investigate this topic and other potential approximations for accelerating Al'brekht's method.

\bibliography{references}
\bibliographystyle{IEEEtran}

\appendix
\subsection{Kronecker Product Identities}\label{sec:identities}
\begin{table}[H]
    \caption{Relevant Kronecker product identities. }
    \-\hspace{-2pt}\makebox[\columnwidth]{\rule{.95\columnwidth}{0.1em}}
    \vspace{-7pt}

    \-\hspace{-2pt}\makebox[\columnwidth]{\rule{.95\columnwidth}{0.1em}}

    \begin{identities}
        \setlength{\itemsep}{3pt}
        \setlength\itemindent{14pt}
        \item \label{ID:T2.4} $\quad (\bA \otimes \bB)(\bD \otimes \bG) = \bA \bD \otimes \bB \bG $
        \item \label{ID:T2.5} $\quad \bA \otimes \bB = \perm{s}{p} (\bB \otimes \bA)\perm{q}{t}$
        \item \label{ID:T2.17} $\quad (\bI_p \otimes \bx)\bA = \bA \otimes \bx$
        \item \label{ID:T2.13} $\quad \Vec[\bA \bD \bB] = (\bB^\top \otimes \bA) \Vec[\bD]$
        \item \label{ID:T3.4} $\quad \begin{aligned}[t]
                \Vec[\bA \bD] & = (\bI_s \otimes \bA) \Vec[\bD] \\
            \end{aligned}$
        \item \label{ID:nonsymmetric-quadraticForm} $\quad \bu^\top \bB \bx  = \Vec[\bB]^\top (\bx \otimes \bu) $
        \item \label{ID:vecKron1} $\quad \Vec\left[\bx^\top \otimes \bI_m\right] = (\bx \otimes \Vec[\bI_m])$
        \item \label{ID:vecKron3} $\quad \Vec\left[\bA \otimes \bB\right] =  \left(\bI_q \otimes \perm{p}{t} \otimes \bI_s \right)\left(\Vec[\bA] \otimes \Vec[\bB] \right)$
    \end{identities}

    \vspace{5pt}

    \-\hspace{9pt}\textit{Dimensions of matrices used in the Kronecker product identities}\\

    \vspace{-2pt}
    \centering
    \begin{tabular}{ccc}
        $\bA(p \times q)$ & $\bD(q \times s)$ & $\bu(s \times 1)$ \\
        $\bB(s \times t)$ & $\bG(t \times u)$ & $\bx(t \times 1)$ \\
    \end{tabular}

    \vspace{4pt}

    \-\hspace{-2pt}\makebox[\columnwidth]{\rule{.95\columnwidth}{0.1em}}
    \label{tab:identities}
\end{table}

\subsection{Proof of \cref{thm:wiPoly}}\label{sec:thm-proof}
Inserting the polynomial forms of $\bf(\bx)$ and $\bg(\bx)$ from~\cref{eq:FOM-Poly}
into the HJB PDE~\cref{eq:HJB-PDE} gives
\begin{equation}\label{eq:HJB-PDE-plugged-in}
    \begin{split}
        0  = & \frac{\partial V^\top(\bx)}{\partial \bx} \left[\bA \bx + \sum_{p=2}^\ell \F_p \kronF{\bx}{p}\right]                                             \\
             & - \frac{1}{2} \frac{\partial V^\top(\bx)}{\partial \bx}\left[\sum_{p=1}^\ell \G_p \left(\kronF{\bx}{p} \otimes \bI_m \right) + \bB\right] \times \\
             & \bR^{-1}\left[\sum_{q=1}^\ell  \left({\kronF{\bx}{q}}^\top \otimes \bI_m \right)\G_q^\top + \bB^\top\right] \frac{\partial V(\bx)}{\partial \bx} \\
             & \qquad + \frac{1}{2} \bx^\top \bQ \bx + \frac{1}{2} \sum_{p=3}^{\lambda} \bq_p^\top \kronF{\bx}{p}
    \end{split}
\end{equation}
The gradient of the value function \cref{eq:vi-coeffs} in Kronecker product form is
\begin{equation}\label{eq:Past-value-deriv}
    \begin{split}
         & \frac{\partial^\top V(\bx)}{\partial \bx}
        = \frac{1}{2}\left (2\bx^\top \bV_2 \right.                                                                                                        \\
         & + \bv_3^\top (\bI_n \otimes \bx \otimes \bx) +  \bv_3^\top (\bx \otimes \bI_n \otimes \bx) + \bv_3^\top (\bx \otimes \bx \otimes \bI_n)         \\
         & + \bv_4^\top (\bI_n \otimes \bx \otimes \bx \otimes \bx) + \bv_4^\top ( \bx \otimes \bI_n \otimes \bx \otimes \bx)                              \\
         & + \left. \bv_4^\top ( \bx \otimes \bx \otimes \bI_n \otimes  \bx) + \bv_4^\top ( \bx \otimes  \bx \otimes \bx \otimes \bI_n )+ \cdots \right ),
    \end{split}
\end{equation}
where without loss of generality we assume that $\bV_2$ is symmetric.
After inserting the polynomial expressions for the dynamics~\cref{eq:FOM-Poly} and the value functions~\cref{eq:vi-coeffs}, the HJB PDE \cref{eq:HJB-PDE-plugged-in} yields an algebraic equation for each of the coefficients $\bv_i$ for $i=2,3,\dots,d$.
The collection of degree~2 terms in \cref{eq:HJB-PDE-plugged-in} is
\begin{equation*}
    \begin{split}
        0  = & \bx^\top \bV_2  \bA \bx   - \frac{1}{2} \bx^\top \bV_2  \bB  \bR^{-1} \bB^\top \bV_2 \bx + \frac{1}{2} \bx^\top \bQ \bx.
    \end{split}
\end{equation*}
Differentiating twice with respect to $\bx$ yields the algebraic Riccati equation \cref{eq:ARE} for $\bV_2$.

The equation for
$\bv_k$ for $k=3,4,\dots,d$
is obtained by collecting the terms of degree~$k$ in \cref{eq:HJB-PDE-plugged-in}; recall that as a consequence of \cref{thm:albrekht}, these are linear algebraic equations.
Thus we separate the collection of degree~$k$ terms in \cref{eq:HJB-PDE-plugged-in} as $0 =-\textit{LHS} + \textit{RHS}$, where $\textit{LHS}$ denotes the terms containing the unknown $\bv_k$
and $\textit{RHS}$ contains the sum of all of the remaining terms
\footnote{Eventually, $\textit{LHS}$ will be the left-hand-side of the linear algebraic equations for $\bv_k$, and $\textit{RHS}$ will be the right-hand-side, i.e. $\textit{LHS} = \textit{RHS}$.}.
The collection of degree~$k$ terms containing $\bv_k$ is
\begin{align*}
    \textit{LHS} & \coloneqq -\frac{1}{2} \bv_k^\top\left( (  \bI_n \otimes  \kronF{\bx}{k-1}) +  ( \bx \otimes \bI_n \otimes  \kronF{\bx}{k-2})+ \cdots \right) \times \\
                 & \qquad(\bA - \bB\bR^{-1}\bB^\top \bV_2 ) \bx.
\end{align*}
With careful algebraic manipulations, the Kronecker products can be reordered to rewrite this collection of terms using the \textit{$k$-way Lyapunov matrix}
\begin{align}
    \textit{LHS}
     & =  -\frac{1}{2} \bv_k^\top \cL_{k}(\bA - \bB\bR^{-1}\bB^\top \bV_2 ) \kronF{\bx}{k} .      \label{eq:k-way-trick-2}
\end{align}

The remaining terms in the HJB PDE, which make up the right-hand-side of the linear algebraic equation for $\bv_k$,
only contain coefficients $\bv_2$ through $\bv_{k-1}$, which have already been computed.
As such, they are symmetric according to \cref{def:sym}.
This allows us to rewrite the gradient of the value function~\cref{eq:Past-value-deriv}
using \cref{ID:T2.5,thm:symPermutation}~as
\begin{align}
    \frac{\partial^\top V(\bx)}{\partial \bx}
     & = \frac{1}{2}\left (2\bv_2^\top (\bI_n \otimes \bx)  + 3\bv_3^\top (\bI_n \otimes \bx \otimes \bx)
    + \cdots \right ) \nonumber                                                                                        \\
     & = \frac{1}{2}\sum_{i=2}^{k-1} i \bv_i^\top (\bI_n \otimes \kronF{\bx}{i-1}).\label{eq:Past-value-deriv-compact}
\end{align}
Rewriting the remaining terms in the HJB PDE \cref{eq:HJB-PDE-plugged-in} with the new gradient of the value function \cref{eq:Past-value-deriv-compact},
the $\textit{RHS}$ terms can be written as
\small
\begin{align*}
     & \textit{RHS}  \coloneqq  \frac{1}{2} \left[\sum_{i=2}^{k-1} i \bv_i^\top (\bI_n \otimes \kronF{\bx}{i-1})\right] \left[\sum_{p=2}^\ell \F_p \kronF{\bx}{p}\right] + \frac{1}{2} \sum_{p=3}^{\lambda} \bq_p^\top \kronF{\bx}{p}
\end{align*}
\vspace*{-1.5em}
\begin{subequations}\label{eq:HJB-PDE-plugged-in-symmetric}
    \begin{align}
         & \label{eq:HJB-PDE-plugged-in-symmetric-1}                                                                                                                                                                                            \\
         & - \frac{1}{8} \left[\sum_{i=2}^{k-1} i \bv_i^\top (\bI_n \otimes \kronF{\bx}{i-1})\right] \left[\sum_{p=1}^\ell \G_p \left(\kronF{\bx}{p} \otimes \bI_m \right) + \bB\right] \times  \label{eq:HJB-PDE-plugged-in-symmetric-2}       \\
         & \bR^{-1}\left[\sum_{q=1}^\ell  \left({\kronF{\bx}{q}}^\top \otimes \bI_m \right)\G_q^\top + \bB^\top\right] \left[\sum_{j=2}^{k-1} (\bI_n \otimes {\kronF{\bx}{j-1}}^\top) \bv_j j\right]. \label{eq:HJB-PDE-plugged-in-symmetric-3}
    \end{align}
\end{subequations}
\normalsize
Again using careful algebraic manipulations, these terms can be manipulated to a more desirable form; specifically, the goal is to put these in the form $\left(\cdot\right)\kronF{\bx}{k}$ so that we may match coefficients of the same polynomial degree.
Beginning with the terms from \cref{eq:HJB-PDE-plugged-in-symmetric-1} containing $\F_p$
\begin{align*}
     & \frac{1}{2}i \bv_i^\top (\bI_n \otimes \kronF{\bx}{i-1}) \F_p \kronF{\bx}{p}, &  & \text{with} &  & p + i - 1 = k,
\end{align*}
a similar trick to that used in \cref{eq:k-way-trick-2} allows us to rewrite this using the \textit{$k$-way Lyapunov matrix} as
\begin{align}\label{eq:N-terms}
    \frac{1}{2}  \bv_i^\top \cL_i(\F_p) \kronF{\bx}{k}.
\end{align}
The terms containing $\bq_p$ are already in the desired form
\begin{align}\label{eq:Q-terms}
     & \frac{1}{2}  \bq_p^\top \kronF{\bx}{p} &  & \text{with} &  & p = k.
\end{align}
The terms containing $\bB$ are
\begin{align}\label{eq:vi-B-terms-1}
     & \frac{1}{8} i \bv_i^\top (\bI_n \otimes \kronF{\bx}{i-1}) \bB\bR^{-1} \bB^\top (\bI_n \otimes {\kronF{\bx}{j-1}}^\top) \bv_j j,
\end{align}
with $i+j-2=k$.
After algebraic manipulations, these terms are written as
\begin{align}\label{eq:B-terms}
     & \frac{1}{8} ij \Vec[\bV_i^\top \bB\bR^{-1} \bB^\top \bV_j]^\top {\kronF{\bx}{k}}.
\end{align}

The remaining terms, which are those containing $\G_p$, are
\small
\begin{align}\label{eq:vi-G-terms-1}
     & \frac{1}{8} i \bv_i^\top (\bI_n \otimes \kronF{\bx}{i-1}) \G_p (\kronF{\bx}{p} \otimes \bI_m)  \times                             \\
     & \hspace{2.5cm} \bR^{-1} ({\kronF{\bx}{q}}^\top \otimes \bI_m)\G_q^\top (\bI_n \otimes {\kronF{\bx}{j-1}}^\top) \bv_j j, \nonumber
\end{align}
\normalsize
with $p \in \left[0, o\right]$,  $o \in \left[1,2\ell\right]$, $q = o - p$, and $i + j + o = k + 2$.
Here the manipulations are much more involved, but ultimately they just require careful application of the identities in \cref{tab:identities}.
After these manipulations, the terms containing $\G_p$ are written as
\small
\begin{align}
     & \frac{1}{8} i j \Vec \Biggl[ \left(\bI_{n^p} \otimes \Vec\left[ \bI_m\right]^\top\right) \left(\Vec\left[\G_q^\top \bV_j \right]^\top \otimes  \right.  \label{eq:G-terms}                                   \\
     & \quad \left. \vphantom{\Vec\left[\G_q^\top \bV_j \right]^\top} \left(\G_p^\top \bV_i \otimes \bR^{-1}\right) \right) \left(\bI_{n^{j-1}} \otimes \perm{n^{i-1}}{n^q m} \otimes \bI_m \right)\times \nonumber \\
     & \hspace{4cm}   \left( \bI_{n^{k-p}} \otimes \Vec\left[ \bI_m\right] \right)\Biggr]^\top {\kronF{\bx}{k}}. \nonumber
\end{align}
\normalsize

The quantities \cref{eq:N-terms,eq:Q-terms,eq:B-terms,eq:G-terms} represent single degree~$k$ terms containing $\F_p,\bq_p,\bB,$ and $\G_p$ contributions, respectively.
Introducing summations over the appropriate combinations of indices, the HJB PDE \cref{eq:HJB-PDE-plugged-in}, written now as $\textit{LHS} = \textit{RHS}$, can be expanded as
\small
\begin{align*}
     & -\frac{1}{2} \bv_k^\top \cL_{k}(\bA - \bB\bR^{-1}\bB^\top \bV_2 ) \kronF{\bx}{k} =
    \frac{1}{2} \sum_{\mathclap{\substack{i,p\geq 2                                                                                                 \\       i + p = k+1}}} \bv_i^\top \cL_i(\F_p) \kronF{\bx}{k} \nonumber \\
     & \hspace{.4cm}+   \frac{1}{2}  \bq_k^\top \kronF{\bx}{k} - \frac{1}{8}\!\!\!\!\sum_{\substack{i,j>2                                           \\ i+j=k+2}} \!\!\!\!ij~{\Vec}(\bV_i^\top \bB \bR^{-1} \bB^\top \bV_j)\kronF{\bx}{k} \\
     & - \frac{1}{8} \sum_{o=1}^{2\ell}\left( \sum_{\substack{p,q \geq 0                                                                            \\p+q=o}} \left(  \sum_{\substack{i,j\geq 2 \\       i+j=k-o+2}} \!\!\!\!ij~\Vec \Biggl[ \left(\bI_{n^p} \otimes \Vec\left[ \bI_m\right]^\top\right) \times  \right.\right. \nonumber\\
     & \hspace{0.25cm} \left(\Vec\left[\G_q^\top \bV_j \right]^\top \otimes \left(\G_p^\top \bV_i \otimes \bR^{-1} \right) \right) \times \nonumber \\
     & \hspace{0.25cm} \left.\left.\left(\bI_{n^{j-1}} \otimes \perm{n^{i-1}}{n^q m} \otimes \bI_m \right)
        \left( \bI_{n^{k-p}} \otimes \Vec\left[ \bI_m\right] \right)\Biggr]^\top
    \vphantom{\sum_{\substack{p,q \geq 0                                                                                                            \\p+q=o}}} \right)\right){\kronF{\bx}{k}}
\end{align*}
\normalsize
We require this to hold for all $\bx$.
Pulling out the factor of ${\kronF{\bx}{k}}$ from every term, multiplying by negative two, and transposing the entire equation results in \cref{eq:LinSysForWk-Poly} and completes the proof for \cref{thm:wiPoly}. \hfill$\blacksquare$
\end{document}